\documentclass[a4paper, 10pt]{amsart}

\usepackage{amssymb}
\usepackage{amscd}
\usepackage{color, hyperref}

\newtheorem{theorem}{Theorem}
\newtheorem{lemma}[theorem]{Lemma}

\newtheorem{remark}[theorem]{Remark}
\newtheorem{proposition}[theorem]{Proposition}
\newtheorem{corollary}[theorem]{Corollary}

\newcommand{\AAMP}{\rm AAMP}
\newcommand{\N}{\mathbb N}
\newcommand{\Z}{\mathbb Z}
\newcommand{\red}{{\text{\rm red}}}
\newcommand{\DP}{\negthinspace :\negthinspace}

\DeclareMathOperator{\Pic}{Pic} 
 
 \DeclareMathOperator{\Cl}{Cl}
\DeclareMathOperator{\spec}{spec}

\begin{document}

\title[On Dedekind domains]{On Dedekind domains whose  class groups are direct sums of cyclic groups}

\author[G.W. Chang and A. Geroldinger]{Gyu Whan Chang and Alfred Geroldinger}

\address{Department of Mathematics Education, Incheon National University,
Incheon 22012, Korea} \email{whan@inu.ac.kr}

\address{University of Graz, NAWI Graz, Institute of Mathematics and Scientific Computing, Heinrichstraße 36,
8010 Graz, Austria}
\email{alfred.geroldinger@uni-graz.at}
\urladdr{https://imsc.uni-graz.at/geroldinger}



\begin{abstract}
For a given family $(G_i)_{i \in \N}$  of finitely generated abelian groups, we construct a Dedekind domain $D$ having the following properties.
\begin{enumerate}
\item $\Pic(D) \cong \bigoplus_{i \in \N}G_i$.
\item For each $i \in \N$, there exists a submonoid $S_i \subseteq D^{\bullet}$ with $\Pic (D_{S_i}) \cong G_i$.
\item Each  class of  $\Pic (D)$ and of all $\Pic (D_{S_i})$ contains  infinitely many prime ideals.
\end{enumerate}
Furthermore, we study orders as well as sets of lengths in the Dedekind domain $D$ and in all its localizations $D_{S_i}$.
\end{abstract}

\maketitle

\section{Introduction}
\smallskip

Claborn's Realization Theorem (\cite[Theorem 7]{Cl66}) states that every abelian group is isomorphic to the class group of a Dedekind domain. This result gave rise to a lot of further research (\cite{Ea-He73}, \cite{Gr74b}, \cite{Sk76}, \cite{Mi-St86}, \cite{Ge-HK92a}). One strand of research was to ask for additional properties of the Dedekind domains with given class group. Leedham-Green (\cite{Le72a}) proved that every abelian group is isomorphic to the class group of a Dedekind domain that is the quadratic extension of a principal ideal domain. A further strand of research asked for the realization of class groups either within special classes of Dedekind domains or within more general classes of Dedekind and Krull domains (\cite{Cl09a,Sm17b, Ch21a, Ch22a}). Nevertheless, there is an abundance of open problems. To mention a classic question, it is still unknown whether every finite abelian group is isomorphic to the class group of the ring of integers of a number field (e.g., \cite{Pe99, Co79a,So83a}). A further direction deals with the distribution of prime divisors in the classes.  Let $G$ be an abelian group and let $(m_g)_{g \in G}$ be a family of cardinal numbers. The question is whether or not there is a monoid or domain (within the given class) whose class group is isomorphic to $G$ and the cardinality of prime divisors in class $g \in G$ is equal to $m_g$ for all $g \in G$. This question has been answered for Krull monoids (\cite[Theorem 2.5.4]{Ge-HK06a}) and for Dedekind domains whose class group has a denumerable generating set (\cite{Gi-He-Sm96}). However, the question is still open for general Dedekind domains (\cite[Section 3.7]{Ge-HK06a}).

The starting point for the present paper is a  realization result by Chang for class groups of almost Dedekind domains (\cite[Theorem 3.5]{Ch22a}).
It is well known that the set of isomorphism classes of finitely generated abelian groups is countable,
so a careful reading of the proof of \cite[Theorem 3.5]{Ch22a}
shows that the following theorem  holds true.

\vspace{.12cm}
\noindent
{\bf Theorem A.}
{\it Let $(G_i)_{i \in \N}$ be a  family of finitely generated abelian groups. Then
there is a Bezout overring $R$ of $\mathbb{Z}[X]$ with the following properties.
\begin{enumerate}
\item $R \cap \mathbb{Q}[X]$ is an almost Dedekind domain.
\item $\Pic (R \cap \mathbb{Q}[X]) \cong \bigoplus_{i \in \N}G_i$.
\item For each $i \in \N$, there is a submonoid $S_i \subseteq \mathbb{Z}^{\bullet}$
such that $R_{S_i} \cap \mathbb{Q}[X]$ is a Dedekind domain with $\Pic (R_{S_i} \cap \mathbb{Q}[X]) \cong G_i$.

\item Every ideal of $R \cap \mathbb{Q}[X]$, that is not contained in $X\mathbb{Q}[X] \cap R$, is invertible.
\end{enumerate}}

\smallskip
\noindent
Note that, whenever we consider a family $(G_i)_{i \in \N}$ of abelian groups,
we neither require that the groups $G_i$ are distinct or non-isomorphic nor that they are nontrivial.

Motivated by this result we establish the following realization result for class groups of Dedekind domains
and this is the main result of the present paper (see Theorem \ref{con2}).

\vspace{.12cm}
\noindent
{\bf Theorem B.}
{\it Let $(G_i)_{i \in \N}$ be a  family of finitely generated abelian groups. Then there is a Dedekind domain $D$ with the following properties.
\begin{enumerate}
\item  $\Pic(D) \cong \bigoplus_{i \in \N}G_i$.

\item For each $i \in \N$,  there is a submonoid $S_i \subseteq D^{\bullet}$ such that $\Pic(D_{S_i}) \cong G_i$.

\item Each  class of  $\Pic (D)$ and of all $\Pic (D_{S_i})$ contains infinitely many prime ideals.
\end{enumerate}}

\smallskip
\noindent
The Dedekind domain $D$, occurring in Theorem B,  is not constructed as an overring of $\mathbb Z[X]$ (unlike Theorem A).
Since every bounded abelian group is a direct sum of cyclic groups, all countably generated, bounded abelian groups occur
as class groups of Dedekind domains with the properties of Theorem B.

\vspace{.12cm}

In Section 2, we briefly discuss what we need of the ideal theory of monoids and domains.
In Section 3, we first construct a reduced Krull monoid $M$ with the properties (1) and (2)
of Theorem B (Corollary \ref{coro3}) for $\Cl_v(M)$ and $\Cl_v(M_{S_i})$.
Then we show that, for a field $K$, the monoid algebra $K[M]$ is a Krull domain with the properties (1), (2) and (3) of Theorem B (Proposition \ref{con1}).
Finally, we  construct a Dedekind domain with the properties of Theorem B (Theorem \ref{con2}).

\smallskip
In Section 4, we study orders in the Dedekind domain $D$ and in its localizations and we study sets of lengths in these domains (Corollary \ref{coro-8} and Corollary \ref{coro-9}). Sets of lengths  depend not only on the respective Picard groups but also on the distribution of prime ideals in the classes, as given in (3) of Theorem B (property (3) is established in Lemma \ref{lemma3}, which is based on the recent paper \cite{Fa-Wi22b}).

\section{Background on the  ideal theory of monoids and domains}

We gather some basic concepts of ideal theory of monoids and domains and fix our notation. Detailed presentations can be found in the monographs (\cite{Gi72a, HK98, Ge-HK06a}).

\subsection{Monoids}
By a {\it semigroup}, we mean a commutative semigroup with identity and
by a {\it monoid}, we  mean a  cancellative semigroup.
Let $M$ be a monoid.  Then $M^{\times}$ denotes its group of units and $\mathsf q (M)$ denotes its quotient group. We say that $M$ is reduced if $M^{\times} = \{1\}$ and
$M_{\red} = \{ a M^{\times} \mid a \in M\}$ is the associated reduced monoid of $M$.
Furthermore, $M$ is called torsionless (or torsionfree) if $a^n = b^n$, where $a, b \in M$ and $n \in \N$, implies that $a=b$. Let $S \subseteq M$ be a submonoid. Then $S$ is said to be divisor-closed if $a, b \in M$ and $ab \in S$ implies that $a \in S$ and $b \in S$. Let
\[
\widehat M = \{ x\in \mathsf q (M) \mid \text{there is a $c\in M$ such that $cx^n\in M$ for every $n\in \N$} \}
\]
denote the {\it complete integral closure} of $M$, and we say that $M$ is {\it completely integrally closed} if $M = \widehat M$.
We denote by
\[
M_S = S^{-1}M = \Big\{ \frac{m}{s} \mid m \in M, s \in S \Big\} \subseteq \mathsf q (M)
\]
the localization of $M$ by $S$. If $T$ denotes the smallest divisor-closed submonoid generated by $S$, then $S^{-1}M= T^{-1}M$.

For the convenience of the reader, we give a brief introduction to the terminologies related to ideal system (details can be found in \cite{HK98, Ge-HK06a, HK11b}).  An {\it ideal system} on a monoid $M$ is a map $r\colon\mathcal P(M)\to\mathcal P(M)$, with $\mathcal P (M)$ denoting the power monoid of $M$, such that the following conditions are satisfied for all subsets $X,Y\subseteq M$ and all $c\in M$:
\begin{itemize}
\item $X\subseteq X_r$.

\item $X\subseteq Y_r$ implies $X_r\subseteq Y_r$.

\item $cM\subseteq \{c\}_r$.

\item $cX_r=(cX)_r$.
\end{itemize}
Let $r$ be an ideal system on $M$. A subset $I \subseteq M$ is called an $r$-ideal if $I_r=I$. We denote by $\mathcal I_r (M)$ the set of all nonempty $r$-ideals, and we define $r$-multiplication by setting $I \cdot_r J = (IJ)_r$ for all $I,J \in \mathcal I_r (M)$. Then $\mathcal I_r (M)$ together with $r$-multiplication is a reduced semigroup with identity element $M$. Let $\mathcal F_r (M)$ denote the semigroup of fractional $r$-ideals, $\mathcal F_r (M)^{\times}$ the group of $r$-invertible fractional $r$-ideals, and $\mathcal I_r^*(M) = \mathcal F_r^{\times} (M) \cap \mathcal I_r (M)$ the  monoid of $r$-invertible $r$-ideals of $M$ with $r$-multiplication.

We will need the $v$-system (in other words, the system of divisorial ideals), the $t$-system, and the $s$-system. To recall the definitions, consider two subsets $X, Y \subseteq \mathsf q(M)$. Then  $(X\DP Y)=\{x\in\mathsf q(M)\mid xY \subseteq X\}$, $X_s = XM$,  $X^{-1}=(M \DP X)$, $X_v=(X^{-1})^{-1}$, and  $X_t=\bigcup_{E\subset X,|E|<\infty} E_v$.

We denote by $\mathfrak X (M)$ the set of nonempty minimal prime $s$-ideals and note that $\mathfrak X (M) \subset t$-$\spec (M)$. The cokernel of the group homomorphism $\mathsf q (M) \to \mathcal F_r^{\times} (M)$, defined by $a \mapsto aM$, is called the $r$-class group of $M$ and is denoted by $\Cl_r (M)$. For an $r$-ideal $I \in \mathcal F_r^{\times} (M)$, we denote by $[I]=[I]_r \in \Cl_r (M)$ the class containing $I$. The class group will be written additively, whence $[I \cdot_r J] = [I]+[J]$ for all $I, J \in \mathcal F_r^{\times} (M)$, and the elements of $\Cl_r (M)$ are considered as subsets of $\mathcal F_r^{\times} (M)$.  In particular, $\boldsymbol 0 = [M] \in \Cl_r (M)$ is the zero element of $\Cl_r (M)$.

The monoid $M$ is a {\it Krull monoid} if it satisfies one of the following equivalent conditions (\cite[Chapter 22.8]{HK98}).
\begin{itemize}
\item $M$ is completely integrally closed and $v$-noetherian (i.e., the ascending chain condition on $v$-ideals holds).

\item Every $t$-ideal is $t$-invertible (i.e., $\mathcal I_t (M) = \mathcal I_t^* (M)$).

\item Every $\mathfrak p \in \mathfrak X (M)$ is $t$-invertible.
\end{itemize}
If $M$ is a Krull monoid, then $v=t$ and the monoid $\mathcal I_v^* (M)$ of $v$-invertible $v$-ideals is a free abelian monoid with $v$-multiplication as operation and with basis $\mathfrak X (M)$. The monoid $M$ is Krull if and only if $M_{\red}$ is Krull. If this holds, then
$\Cl_t (M) = \Cl_v (M) = \Cl_v (M_{\red}) = \Cl_t (M_{\red})$. Obviously, reduced Krull monoids are torsionless.

\subsection{Integral domains} By a {\it domain}, we mean a commutative integral domain with identity. Let $D$ be a domain. Then $D^{\times}$ denotes its unit group, $\mathsf q (D)$ the quotient field of $D$, and $D^{\bullet} = D \setminus \{0\}$ denotes its monoid of nonzero elements. If $\mathfrak p \in \spec (D)$, then $D \setminus \mathfrak p \subset D^{\bullet}$ is a divisor-closed submonoid.  For an ideal system $r$ of $D$,  $\mathcal I_r (D)$, $\mathcal I_r^* (D)$, and $\mathcal F_r^{\times} (D) = \mathsf q (\mathcal I_r^* (D))$ denote the respective ideal semigroups as in the monoid case. For the $d$-system (the system of usual ring  ideals), we omit the subscripts, whence $\mathcal I (D)$ denotes the semigroup of nonzero ideals of $D$, $\mathcal I^* (D)$ the monoid of invertible ideals of $D$, and $\mathcal F^{\times} (D) = \mathsf q (\mathcal I^* (D))$ the group of invertible fractional ideals of $D$.
Divisorial ideals and $t$-ideals of $D^{\bullet}$ and of $D$ are in one-to-one correspondence. Indeed, if $r=t$ or $r=v$, then the maps
\[
\iota^\bullet \colon
\begin{cases}
\mathcal F_r(D) &\to \  \mathcal F_r(D^\bullet)\\
\quad \mathfrak a &\mapsto \quad \mathfrak a \setminus \{0\}
\end{cases}
\qquad \text{and} \qquad
\iota^\circ \colon
\begin{cases}
\mathcal F_r(D^\bullet) &\to \  \mathcal F_r(D)\\
\quad \mathfrak a &\mapsto \quad \mathfrak a \cup \{0\}
\end{cases}
\]
are semigroup isomorphisms, inverse to each other,  mapping
$\mathcal I_r(D)$ onto $\mathcal I_r(D^\bullet)$ and fractional
principal ideals of $D$ onto fractional principal ideals of
$D^\bullet$. This easily implies that $\Cl_r (D) = \Cl_r (D^{\bullet})$, and that $D$ is a Krull domain if and only if $D^{\bullet}$ is a Krull monoid. If $r=d$ is the system of usual ring ideals, then
\[
\Pic(D) = \Cl_d (D) \subseteq \Cl_t(D) \subseteq \Cl_v(D)
\]
is  the {\it Picard group}  of $D$, which is a subgroup of the $t$-class group $\Cl_t(D)$. If $D$ is a one-dimensional  domain, then $\Cl_t (D) =  \Pic (D)$. If $D$ is $v$-noetherian, then $v=t$, whence $\Cl_v (D)= \Cl_t (D)$, and if $D$ is a Krull domain, then $\Cl_v (D)$ is (isomorphic to) the  divisor class group of $D$.

The domain $D$ is a Dedekind domain if and only if $D$ is a Krull domain with Krull dimension at most one.
Furthermore, the following are equivalent (see (\cite[Chapter 23]{HK98}, and \cite{Lo06a} for a survey on almost Dedekind domains).
\begin{itemize}
\item[(a)] $D$ is an almost Dedekind domain.

\item[(b)] $\mathcal I (D)$ is cancellative.

\item[(c)] $D_M$ is a DVR for all nonzero maximal ideals $M$ of $D$.
\end{itemize}
Thus, every almost Dedekind domain is completely integrally closed and of Krull dimension at most one.
Hence, a Dedekind domain is exactly a noetherian almost Dedekind domain,
or equivalently, a Krull almost Dedekind domain.
We will also use that  if each maximal ideal of $D$ is a $t$-ideal (e.g., if $D$ is an almost Dedekind domain),
then a $t$-invertible ideal of $D$ is invertible.

\section{On the  construction of the desired Dedekind domain}

First, we construct a Krull domain $D$,
which has the properties (1), (2), and (3) of Theorem B for $\Cl_v(D)$ and $\Cl_v(D_{S_i})$.
This domain  will be constructed as a monoid algebra $K[M]$
of a Krull monoid $M$ over a field $K$, and the monoid $M$ can be chosen
as  the associated reduced monoid of the complement of $X\mathbb{Q}[X] \cap R$ of Theorem A (Corollary \ref{coro3}). We then use $D$ to
prove the existence of a Dedekind domain with the properties of Theorem B (Theorem \ref{con2}).
We begin this section with
some ideal-theoretic properties of the desired monoid.

\smallskip
\begin{lemma} \label{lemma1}
Let $D$ be a domain, $Q$ be a prime ideal of $D$, $M = D \setminus Q \subset D^{\bullet}$,
$I \in \mathcal I (D)$ with $I \cap M \neq \emptyset$,
and $J$ be a fractional $s$-ideal of $M$. Then the following statements hold.
\begin{enumerate}
\item $I = (I \cap M)D$.
\item $(JD)^{-1} = J^{-1}D$.
\item $(JD)_v = J_vD$.
\item $(JD)_t = J_tD$.
\item $JD$ is $t$-invertible if and only if $J$ is $t$-invertible.
\end{enumerate}
\end{lemma}

\begin{proof}
(1) Clearly, $(I \cap M)D \subseteq I$. For the reverse containment,
let $x \in I$. We may assume that $x \not\in I \cap M$, so if
$s \in I \cap M$, then $s+x \in I \cap M$ because $x \in Q$. Hence, $x \in (s, x+s)D \subseteq (I \cap M)D$.
Thus, $I \subseteq (I \cap M)D$.

(2) Note that $(J^{-1}D)(JD) \subseteq MD = D$, so $J^{-1}D \subseteq (JD)^{-1}$.
For the reverse containment, choose $s \in J$, so $s(JD)^{-1} \subseteq D$ and $s(JD)^{-1} \cap M \neq \emptyset$.
Then, by (1), $s(JD)^{-1} = (s(JD)^{-1} \cap M)D$, and thus $(JD)^{-1} \subseteq ((JD)^{-1} \cap \mathsf q (M))D$.
Now, if $y \in (JD)^{-1} \cap \mathsf q (M)$, then $yJ \subseteq y(JD) \subseteq D$, whence
$yJ \subseteq D \cap \mathsf q (M) = M$, and so we have  $y \in J^{-1}$. Hence,
$(JD)^{-1} \cap \mathsf q (M) \subseteq J^{-1}$, which implies that $(JD)^{-1} \subseteq J^{-1}D$.

(3) $(JD)_v = ((JD)^{-1})^{-1} = ((J^{-1})D)^{-1} = ((J^{-1})^{-1})D = J_vD$ by (2).

(4)  By (3), we have
\begin{eqnarray*}
(JD)_t &=& \bigcup \ \{B_v \mid B \subseteq JD \text{ and } 0 \ne B \ \text{is a finitely generated fractional ideal of $D$}\} \\
    &=& \bigcup \ \{(CD)_v \mid C \text{ is a finitely generated $s$-ideal of $M$ with} \ C \subseteq J\}\\
    &=& \bigcup \ \{C_vD \mid C \text{ is a finitely generated $s$-ideal of $M$ with} \ C \subseteq J\} \,, \\
    &=& J_tD,
\end{eqnarray*}
whence $(JD)_t = J_tD$.

(5) By (2) and (4), $((JD)(JD)^{-1})_t = (JJ^{-1})_tD$.
Hence, if $(JJ^{-1})_t = M$, then $((JD)(JD)^{-1})_t = D$.
Conversely, assume that $((JD)(JD)^{-1})_t = D$. Then there is a finitely generated subideal
$A$ of $(JJ^{-1})_t$ such that $(AD)_v = A_vD = D$. Now, let $x \in A_vD \cap M$.
Then, by (2), $$xA^{-1} \subseteq xA^{-1}D \cap M = x(AD)^{-1} \cap M \subseteq D \cap M = M,$$
so $x \in A_v$. Thus, $A_vD \cap M = A_v$, which implies that $(JJ^{-1})_t = M$.
\end{proof}

\begin{remark}
{\em Lemma \ref{lemma1}(1) need not be true if $I$ is not an integral ideal of $D$.
For example, if $a \in Q \setminus \{0\}$ and $s \in M$,
then $I = \frac{s}{a}D$ is a fractional ideal of $D$, $I \cap \mathsf q (M) \neq \emptyset$,
but $I \neq (I \cap \mathsf q (M))D$. }
\end{remark}


The next result is an almost Dedekind domain analog of the fact that if $D$ is a Dedekind domain, then
$D^{\bullet}$ is a Krull monoid with $\Pic (D) \cong \Cl_v(D^{\bullet})$.

\begin{proposition} \label{lemma2}
Let $D$ be an almost Dedekind domain, $Q$ be a nonzero prime ideal of $D$,
and $M = D \setminus Q \subseteq D^{\bullet}$.
Assume that each ideal of $D$, that is not contained in $Q$, is invertible.
Then $M$ has the following properties.
\begin{enumerate}
\item $M$ is a Krull monoid.
\item $\Pic (D) \cong \Cl_v (M)$.
\item If $S \subseteq M$ is a submonoid, then $M_S$ is a Krull monoid and $\Pic (D_S) \cong \Cl_v (M_S)$.
\end{enumerate}
\end{proposition}

\begin{proof}
(1) We need to show that every $t$-ideal of $M$ is $t$-invertible. Let $J$ be a $t$-ideal of $M$.
Then $JD \nsubseteq Q$, and hence $JD$ is invertible.
Since an invertible ideal is $t$-invertible, $J$ is $t$-invertible by Lemma \ref{lemma1}(5).
Thus, $M$ is a Krull monoid.

(2)  Let $\varphi : \mathcal F_t^{\times} (M) \to \mathcal F_t^{\times} (D)$ be a map defined by $\varphi (J) = JD$.
Then, by  Lemma \ref{lemma1}(4) and (5), $\varphi$ is well-defined.
Moreover, if $I, J$ are two $t$-invertible fractional $t$-ideals of $M$, then
\[
  \begin{aligned}
    \varphi (I  \cdot_t J) \,\, & = \,\, \varphi \big( (IJ)_t \big) \,\, = \,\, (IJ)_tD
                                                  = \,\, \big( (ID)(JD) \big)_t \,\,\\
                                                  & = \,\, ID \cdot_t JD
                                                 = \,\, \varphi (I) \cdot_t \varphi (J) \,,
  \end{aligned}
\]
whence $\varphi$ is a group homomorphism.
Since principal ideals are mapped onto principal ideals, $\varphi$ induces a group homomorphism
$\widetilde{\varphi} : \Cl_t (M) \to \Cl_t (D)$,
given by $\widetilde{\varphi} \big( [J]  \big) = [JD]$.

Note that  $\Cl_v(M) = \Cl_t(M)$ by (1) and $\Pic(D) = \Cl_t(D)$ because $D$ is one-dimensional.
Hence, it suffices to show that $\widetilde{\varphi}$ is bijective.
For the injectivity of $\widetilde{\varphi}$,
let $I, J$ be two $t$-invertible fractional $t$-ideals of $M$ such that
$\widetilde{\varphi} ( [I] ) = \widetilde{\varphi} ( [J] )$.
Then $ID = u(JD)$ for some $u \in \mathsf q (D)$. Without loss
of generality, we may assume that $u \in \mathsf q (M)$, $I \subseteq M$, and $uJ \subseteq M$.
Thus, by the proof of Lemma \ref{lemma1}(5),
\[
  I \, = \, ID \cap M \, = \, uJD \cap M \, = \, uJ \,.
\]
whence $[I] = [J]$. Thus, $\widetilde{\varphi}$ is injective.
To verify that  $\widetilde{\varphi}$ is surjective, let
$A$ be an invertible ideal of $D$. Then $AA^{-1} \nsubseteq Q$, and hence
there is an $x \in A^{-1}$ such that $xA \nsubseteq Q$ and $xA \subseteq D$.
Note that $[A] = [xA]$, $xA = (xA \cap M)A$, and $xA \cap M$ is $t$-invertible
by Lemma \ref{lemma1}, whence $\widetilde{\varphi} ( [xA \cap M]) = [A]$.
Thus, $\widetilde{\varphi}$ is surjective.
(This proof is similar to the proof of \cite[Theorem 3.6]{Ch-Oh22a}.)

(3) Note that $M_S = D_S \setminus Q_S$, $D_S$ is an almost Dedekind domain,
and each ideal of $D_S$, that is not contained in $Q_S$, is invertible.
Thus, $M_S$ is a Krull monoid and $\Pic (D_S) \cong \Cl_v (M_S)$ by (1) and (2).
\end{proof}

We now use the almost Dedekind domain of Theorem A to
construct a reduced Krull monoid with some preassigned properties on the class group.

\begin{corollary} \label{coro3}
Let $(G_i)_{i \in \N}$ be a  family of finitely generated abelian groups.
Then there is a reduced Krull monoid $M$ with the following properties.
\begin{enumerate}
\item $\Cl_v (M) \cong \bigoplus_{i \in \N}G_i$.
\item For each $i \in \N$, there is a submonoid $S_i \subseteq M$ such that $\Cl_v (M_{S_i}) \cong G_i$.
\end{enumerate}
\end{corollary}

\begin{proof}
By Theorem A, there is an almost Dedekind domain $T$ with the following four properties:
(1) $\mathbb{Z}[X] \subseteq T \subseteq \mathbb{Q}[X]$,
(2) $T$ has a maximal ideal $Q$ so that every ideal of $T$ that is not contained in $Q$ is invertible,
(3) $\Pic (T) \cong \bigoplus_{i \in \N}G_i$, and
(4) for each $i \in \N$, there is a submonoid $S_i \subseteq T \setminus Q$ with $\Pic (T_{S_i}) \cong G_i$.

\vspace{.1cm}
Proposition \ref{lemma2} implies that the monoid  $M =  T \setminus Q \subseteq T^{\bullet}$
is a Krull monoid with $\Pic (T) \cong \Cl_v (M)$, and with $\Pic (T_S) \cong \Cl_v (M_S)$
for any submonoid  $S \subseteq M$.
We assert that the reduced monoid $M_{\red}$ has the desired properties.
Since $\Cl_v (M_{\red}) = \Cl_v (M) \cong \bigoplus_{i \in \N}G_i$, it remains to verify property (2).

\vspace{.1cm}
Let $S \subseteq M$ be a submonoid and without restriction we may suppose that $S$ is divisor-closed. The following formulas are well-known and easy to check:
\[
\begin{aligned}
S^{\times} & = M^{\times}, (S^{-1}M)^{\times}= \mathsf q (S), \mathsf q (S^{-1}M) = \mathsf q (M) , \\
\mathsf q (M_{\red}) & = \mathsf q (M)/M^{\times}, \mathsf q \big( (S^{-1}M)_{\red} \big) = \mathsf q (M)/\mathsf q (S), \quad \text{and} \\
(S^{-1}M)_{\red} & = \big\{ \frac{a}{s} \mathsf q (S) \mid a \in M, s \in S \big\} = \{ a \mathsf q (S) \mid a \in M \} \subseteq \mathsf q (M)/\mathsf q (S) \,.
\end{aligned}
\]
Applying these formulas to the monoids $S_{\red}$ and $M_{\red}$ we obtain that
\[
\begin{aligned}
(S_{\red}^{-1}M_{\red})^{\times} & = \mathsf q (S_{\red}) = \mathsf q (S)/M^{\times}, \\
\mathsf q \big( (S_{\red}^{-1}M_{\red})_{\red} \big) & = \mathsf q (M_{\red})/\mathsf q (S_{\red}) \cong \mathsf q (M)/\mathsf q (S), \quad \text{and} \\
\Big( S_{\red}^{-1}M_{\red}\Big)_{\red} & = \Big\{ \frac{aM^{\times}}{sM^{\times}} \mathsf q (S_{\red}) \mid a \in M, s \in S \Big\} \cong \{ a \mathsf q (S) \mid a \in M\} \subseteq \mathsf q (M)/\mathsf q (S) \,.
\end{aligned}
\]
Thus, the class groups of $S^{-1}M$, of $(S^{-1}M)_{\red}$, and of $\Big( S_{\red}^{-1}M_{\red}\Big)_{\red}$ coincide, whence (2) follows.
\end{proof}

Let $D$ be a domain, $\Gamma$ be an
additive monoid, and let $D[\Gamma]$ be the monoid algebra of
$\Gamma$ over $D$. Then $D[\Gamma]$ is a free $D$-module and we denote its $D$-basis as $\{T^{\gamma} \mid \gamma \in \Gamma \}$. Thus, every element $f \in D[\Gamma]$ can be written uniquely  in the form
\[
f = \sum_{\gamma \in \Gamma} c_{\gamma} T^{\gamma},
\]
where $c_{\gamma} \in D$ for all $\gamma \in \Gamma$  and $c_{\gamma} \neq 0$ for only finitely many $\gamma \in \Gamma$.
Furthermore,  $D[\Gamma]$ is a commutative ring with identity \cite[page 64]{Gi84},
and $D[\Gamma]$ is an integral domain if and only if $\Gamma$ is torsionless \cite[Theorem 8.1]{Gi84}.

\begin{lemma} \label{lemma3}
Let $K$ be a field, $\Gamma$ be a reduced monoid, and $N \subseteq \Gamma$ be a submonoid.
\begin{enumerate}
\item $K[\Gamma]$ is a Krull domain if and only if $\Gamma$ is a Krull monoid.
\item If $\Gamma$ is a Krull monoid, then  $K[\Gamma_N]$ is a Krull domain, $\Cl_v (K[\Gamma_N]) \cong  \Cl_v (\Gamma_N)$,
and each  class of $\Cl_v (K[\Gamma_N])$ contains infinitely many height-one prime ideals.
\end{enumerate}
\end{lemma}

\begin{proof}
(1) This follows  directly from \cite[Theorem 15.6]{Gi84}.

(2) $K[\Gamma]$ is a Krull domain by (1). Now, let $S = \{T^{\gamma} \mid \gamma \in N\}$,
then $S$ is a submonoid of $K[\Gamma]$ and $K[\Gamma]_S = K[\Gamma_N]$.
Thus, the results follow from \cite[Corollary 43.6]{Gi72a},  \cite[Corollary 16.8]{Gi84},
and \cite[Theorem]{Fa-Wi22b}, respectively.
\end{proof}

We now present our first construction of a Krull domain $D$ whose
class group is a direct sum of a given countable family of cyclic groups,
each of which is also the class group of a localization of $D$.

\smallskip
\begin{proposition} \label{con1}
Let $(G_i)_{i \in \N}$ be a  family of finitely generated abelian groups.
Then there is a Krull domain $D$ with the following properties.
\begin{enumerate}
\item $\Cl_v (D) \cong \bigoplus_{i \in \N}G_i$.

\item For each $i \in \N$, there exists a submonoid $S_i \subseteq D^{\bullet}$ such that $\Cl_v (D_{S_i}) \cong G_i$.

\item Each  class of  $\Cl_v (D)$ and of all $\Cl_v (D_{S_i})$ contains infinitely many height-one prime ideals.
\end{enumerate}
Moreover, $D$ can be chosen in such a way that $D/P$ is infinite for all height-one prime ideals $P$ of $D$.
\end{proposition}

\begin{proof}
Let $M$ be the reduced Krull monoid of Corollary \ref{coro3} (whence $M$ is torsionless), let
$K$ be a field, and let $D = K[M]$ be the monoid algebra of $M$ over $K$.
Then $D$ is a Krull domain, $\Cl_v (D) \cong \Cl_v (M) \cong  \bigoplus_{i \in \N}G_i$,
and each  class of $\Cl_v (D)$  contains infinitely many height-one prime ideals by Lemma \ref{lemma3}.
Finally, for each $i \in \N$, there is a submonoid $N_i \subseteq M$ such
that $\Cl_v (M_{N_i}) \cong G_i$. Then $S_i = \{T^{\gamma} \mid \gamma \in N_i\} \subseteq K[M]$ is a submonoid with $D_{S_i} = K[M_{N_i}]$,
$\Cl_v (D_{S_i}) \cong \Cl_v (M_{N_i}) \cong G_i$, and each  class of $\Cl_v (D_{S_i})$ contains infinitely many  height-one prime ideals.
Therefore, $D = K[M]$ is a Krull domain with the properties (1), (2), and (3).

Moreover, assume that $K$ is an infinite field
and let $G = \mathsf q (M)$ be the quotient group of $M$.
If $P$ is a height-one prime ideal of $K[M]$, then
there are two possibilities for $P$.
First, suppose that $P = fK[G] \cap K[M]$, where  $f \in  K[G]$ is a prime element (note that  $K[G]$ is factorial).
Then there is an inclusion $K \hookrightarrow K[M]/P$, whence the factor ring is infinite.
Second, suppose that $P = K[Q]$, where $Q$ is a height-one prime ideal of $M$.
Since $K[M \setminus Q]$ is a set of representatives of the factor ring $K[M]/K[Q]$, we obtain that the factor ring is infinite.
\end{proof}

Let $R$ be an integral domain, $\{X_{\alpha}\}$ be an infinite set of indeterminates over $R$, $T$
be the divisor-closed submonoid of $R[\{X_{\alpha}\}]$ generated by
all nonconstant  polynomials, that are prime elements of $R[\{X_{\alpha}\}]$, and $D = R[\{X_{\alpha}\}]_T$.
Then $\Cl_t (D) \cong \Cl_t (R)$ if and only if $R$ is integrally closed, and
$R$ is a Krull domain if and only if $D$ is a Dedekind domain (\cite[Theorem 3.5]{Ch21a}).
We are now ready to state the main result of this paper.

\smallskip
\begin{theorem} \label{con2}
Let $(G_i)_{i \in \N}$ be a  family of finitely generated abelian groups.
Then there is a Dedekind domain $D$ with the following properties.
\begin{enumerate}
\item  $\Pic(D) \cong \bigoplus_{i \in \N}G_i$.

\item For each $i \in \N$,  there is a submonoid $S_i \subseteq D^{\bullet}$ such that $\Pic(D_{S_i}) \cong G_i$.

\item Each  class of  $\Pic (D)$ and of all $\Pic (D_{S_i})$ contains infinitely many height-one prime ideals.
\end{enumerate}
\end{theorem}

\begin{proof}
Let $R$ be the Krull domain of Proposition \ref{con1}, let $\{X_{\alpha}\}$
be an infinite set of indeterminates over $R$, $T$ be the divisor-closed submonoid of $R[\{X_{\alpha}\}]$ generated by all nonconstant prime polynomials in $R[\{X_{\alpha}\}]$, and
$D = R[\{X_{\alpha}\}]_T$. Then, by  \cite[Theorem 3.5]{Ch21a}, $D$ is a Dedekind domain and
$\Pic (D) = \Cl_v (R) \cong \bigoplus_{i \in \N}G_i$.
Moreover, if $i \in \N$, then $\Cl_v (R_{N_i}) \cong G_i$ for some submonoid $N_i \subseteq R^{\bullet}$. Note that if $g \in T$ is a prime polynomial in $R[\{X_{\alpha}\}]$,
then $g$ is a prime polynomial in $R_{N_i}[\{X_{\alpha}\}]$. Hence, if we let $T_0$
be the divisor-closed submonoid of $R_{N_i}[\{X_{\alpha}\}]$ generated by all nonconstant prime polynomials in $R_{N_i}[\{X_{\alpha}\}]$,
then $T \subseteq T_0$ and
$S_i:= N_iT_0$ is a submonoid of $D = R[\{X_{\alpha}\}]_T$, whence $D_{S_i} = (R_{N_i}[\{X_{\alpha}\}])_{T_0}$.
Therefore, we obtain that $\Pic (D_{S_i}) \cong \Cl_v((R_{N_i}[\{X_{\alpha}\}])_{T_0}) \cong \Cl_v(R_{N_i}) \cong G_i$.

For (3), recall that each  class of $\Cl_v (R[\{X_{\alpha}\}])$
is of the form $\big[ IR[\{X_{\alpha}\}] \big]$ for some ideal $I$ of $R$ \cite[Corollary 2.13]{eik02}
and each $[I] \in \Cl_v (R)$ contains infinitely many height-one prime ideals $P$ of $R$ by Proposition \ref{con1}. Hence,
$\big[ IR[\{X_{\alpha}\}] \Big] = \big[PR[\{X_{\alpha}\}] \big]$ and $PR[\{X_{\alpha}\}]$
is a height-one prime ideal of $R[\{X_{\alpha}\}]$. Next, note that
each  class of $\Cl_v (R[\{X_{\alpha}\}]_T)$ is of the form
$\big[ IR[\{X_{\alpha}\}]_T \big]$, which is equal to $\big[ PR[\{X_{\alpha}\}]_T \big]$
and $PR[\{X_{\alpha}\}]_T$ is a height-one prime ideal. Thus, each  class of $\Cl_v (R[\{X_{\alpha}\}]_T)$
contains infinitely many height-one prime ideals. Finally, note that each class of $\Cl_v (R_{N_i})$ contains infinitely many height-one prime ideals
by Proposition \ref{con1} and $D_{S_i} = R_{N_i}[\{X_{\alpha}\}]_{T_0}$. Therefore,
each  class of $\Pic (D_{S_i})$ contains infinitely many height-one prime ideals.
Thus, $D = R[\{X_{\alpha}\}]_T$ is the desired Dedekind domain.
\end{proof}

\smallskip
Using a very different construction (based on rings of integer-valued polynomials)
Peruginelli \cite{Pe23a} obtained independently a result in the flavor of Theorem \ref{con2}, without Property (3).

\smallskip

\section{Orders and sets of lengths of Dedekind domains}

Let $D$ be a Dedekind domain with quotient field $K$ and let  $\mathcal O \subseteq D$ be a subring.
Then $\mathcal O$ is called an {\it order} in $D$ if $\mathsf q (\mathcal O) = K$ and $D$ is a finitely generated $\mathcal O$-module.
In this section,
we study orders and sets of lengths of the Dedekind domain
occurring in Theorem \ref{con2}. These results heavily depend on Property (3) of Theorem \ref{con2}.

Suppose that $\mathcal O$ is an order in a Dedekind domain $D$. Then $\mathcal O$ is one-dimensional, noetherian, and $\overline{ \mathcal O} = D$, whence $\mathcal O$ is a weakly Krull domain and $\Pic (\mathcal O) = \mathcal \Cl_v (\mathcal O)$.
We study the distribution of height-one prime ideals in the classes of the Picard group. Recall that a class $g \in \Pic ( \mathcal O)$ is considered as a subset of $\mathsf q ( \mathcal I^* ( \mathcal O) )$ and $\mathfrak X (\mathcal O) \cap g$ is the set of height-one prime ideals lying in class $g \in \Pic (\mathcal O)$.

The conductor
\[
\mathfrak f = \{ a \in D \mid a D \subseteq \mathcal O \}
\]
is a non-zero ideal of $D$ (we refer to \cite{Re16a} for a characterization of ideals of $D$ occurring as conductor ideals of some order of $D$), and the monoid $\mathcal O^* = \{a \in \mathcal O^{\bullet} \mid a\mathcal O + \mathfrak f = \mathcal O\}$ is
a Krull monoid with class group $\Cl_v (\mathcal O^*) \cong \Pic (\mathcal O)$ (\cite[Theorem 2.11.12]{Ge-HK06a}).
Next we summarize ideal theoretic properties of $D$, $\mathcal O$, and their relationship. Details and proofs can be found in \cite[Theorem 2.11.12]{Ge-HK06a} and we use the same notation as there. We set
\[
\begin{aligned}
\mathcal I_{\mathfrak f} (D) = \{\overline{\mathfrak a} \in \mathcal I (D) \mid \overline{\mathfrak a} + \mathfrak f = D \} , \quad &
\mathcal I_{\mathfrak f} (\mathcal O) = \{ \mathfrak a \in \mathcal I (\mathcal O) \mid \mathfrak a + \mathfrak f = \mathcal O \} , \\
\mathfrak X_{\mathfrak f} (D) = \mathcal I_{\mathfrak f} (D) \cap \mathfrak X (D) , \quad \text{and} \quad  &
\mathfrak X_{\mathfrak f} (\mathcal O) = \mathcal I_{\mathfrak f} (\mathcal O) \cap \mathfrak X (\mathcal O) \,.
\end{aligned}
\]
Then $\mathfrak X_{\mathfrak f} (\mathcal O)$ is the set of invertible prime ideals of $\mathcal O$. The sets $\mathcal I_{\mathfrak f} (D)$ resp. $\mathcal I_{\mathfrak f} ( \mathcal O)$ are free abelian monoids with usual ideal multiplication and with basis $\mathfrak X_{\mathfrak f} (D)$ resp. $\mathfrak X_{\mathfrak f} (\mathcal O)$. The map
\[
\delta^* \colon \mathcal I_{\mathfrak f} (D) \to \mathcal I_{\mathfrak f} (\mathcal O) , \quad \text{defined by } \quad \overline{\mathfrak a} \longmapsto \overline{\mathfrak a} \cap \mathcal O \,,
\]
is a monoid isomorphism mapping $\mathfrak X_{\mathfrak f} (D)$ onto $\mathfrak X_{\mathfrak f} (\mathcal O)$ and
\[
\overline{\mathfrak a} = ( \overline{\mathfrak a} \cap \mathcal O)D \quad \text{for every} \quad \overline{\mathfrak a} \in \mathcal I_{\mathfrak f} (D) \,.
\]
Moreover, $\delta^*$ induces  an epimorphism
\[
\gamma \colon \Pic (\mathcal O ) \ \to \ \Pic (D), \quad \text{defined by} \quad [\overline{\mathfrak a} \cap \mathcal O] \longmapsto [\overline{\mathfrak a}] \in \Pic (D)
\]
for every $\overline{\mathfrak a} \in \mathcal I_{\mathfrak f} (D)$ (see \cite[Theorem 2.10.9]{Ge-HK06a}).

Our goal is a most careful analysis of the map $\gamma$. To do so we introduce the following subsets of $\Pic ( \mathcal O)$ and of $\Pic (D)$:
\begin{itemize}
\item $G_{\mathfrak f} ( \mathcal O) = \{ [\mathfrak p] \in \Pic (\mathcal O) \mid \mathfrak p \in \mathfrak X_{\mathfrak f} ( \mathcal O)\} \subseteq \Pic (\mathcal O)$ denotes the set of classes of $\Pic (\mathcal O)$ containing invertible prime ideals,

\item $G_{\mathfrak f} (D) = \{ [\mathfrak P] \in \Pic (D) \mid \mathfrak P \in \mathfrak  X_{\mathfrak f} (D) \} \subseteq \Pic (D)$    denotes the set of classes of $\Pic (D)$ containing  prime ideals coprime to the conductor,

\item $G_{\mathfrak f, \infty} ( \mathcal O) = \{ [\mathfrak p] \in \Pic (\mathcal O) \mid \mathfrak p \in \mathfrak X_{\mathfrak f} ( \mathcal O), |\mathfrak X_{\mathfrak f} (\mathcal O) \cap [\mathfrak p]|=\infty \} \subseteq \Pic (\mathcal O)$ denotes the set of classes of $\Pic (\mathcal O)$ containing infinitely many invertible prime ideals, and

\item $G_{\mathfrak f, \infty} (D) = \{ [\mathfrak P] \in \Pic (D) \mid \mathfrak P \in \mathfrak  X_{\mathfrak f} (D), |\mathfrak X_{\mathfrak f} (D) \cap [\mathfrak P]|=\infty  \} \subseteq \Pic (D)$    denotes the set of classes of $\Pic (D)$ containing  infinitely many prime ideals coprime to the conductor.
\end{itemize}

We continue with a lemma.

\smallskip
\begin{lemma} \label{technical}
Let all notation be as above.
\begin{enumerate}
\item The map $\gamma_{\mathfrak f} = \gamma \restriction_{ G_{\mathfrak f} ( \mathcal O)} \colon G_{\mathfrak f} ( \mathcal O) \to G_{\mathfrak f} ( D)$ is surjective, whence $|G_{\mathfrak f} ( D)| \le |G_{\mathfrak f} ( \mathcal O)|$.

\item $\gamma_{\mathfrak f} \big( G_{\mathfrak f, \infty} ( \mathcal O) \big) \subseteq G_{\mathfrak f, \infty} (D)$.

\item If $\gamma_{\mathfrak f}^{-1} (g)$ is finite for every $g \in G_{\mathfrak f, \infty} (D)$, then $\gamma_{\mathfrak f, \infty} = \gamma \restriction_{ G_{\mathfrak f, \infty} ( \mathcal O)} \colon G_{\mathfrak f, \infty} ( \mathcal O) \to G_{\mathfrak f, \infty} ( D)$ is surjective, whence $|G_{\mathfrak f, \infty} ( D)| \le |G_{\mathfrak f, \infty} ( \mathcal O)|$.
\end{enumerate}
\end{lemma}

\begin{proof}
We use that $\delta^* \colon \mathcal I_{\mathfrak f} (D) \to \mathcal I_{\mathfrak f} (\mathcal O)$ is an isomorphism and that $\gamma \colon \Pic (\mathcal O ) \ \to \ \Pic (D)$ is an epimorphism.

(1) To show that $\gamma_{\mathfrak f}$ is surjective, let $g \in G_{\mathfrak f} (D)$ be given, say $g = [\mathfrak P]$ with $\mathfrak P \in \mathfrak X_{\mathfrak f} (D)$. Then $\mathfrak P \cap \mathcal O \in \mathfrak X_{\mathfrak f} ( \mathcal O)$, $[\mathfrak P \cap \mathcal O] \in G_{\mathfrak f} (\mathcal O)$, and $\gamma ( [ \mathfrak P \cap \mathcal O]) = [\mathfrak P]$.

(2)  We assert that, for every $g \in G_{\mathfrak f, \infty} ( \mathcal O)$, the map $\psi_g \colon g \cap \mathfrak X_{\mathfrak f} ( \mathcal O) \to \gamma (g) \cap \mathfrak X_{\mathfrak f} (D)$, defined by $\psi_g (\mathfrak p) = \mathfrak p D$, is injective. If this holds, then $\infty = |g \cap \mathfrak X_{\mathfrak f} ( \mathcal O)| \le |\gamma (g) \cap \mathfrak X_{\mathfrak f} (D)|$, whence $\gamma_{\mathfrak f} (g) \in G_{\mathfrak f, \infty} (D)$.

Let $g \in G_{\mathfrak f, \infty} ( \mathcal O)$. If $\mathfrak p \in g \cap \mathfrak X_{\mathfrak f} ( \mathcal O)$, then $g = [\mathfrak p]$, $\mathfrak p D \in \mathfrak X_{\mathfrak f} (D)$, and $\mathfrak p D \in [\mathfrak p D]  = \gamma ( [\mathfrak p]) = \gamma (g)$, whence $\psi_g$ is well-defined. If $\mathfrak p, \mathfrak p' \in g \cap \mathfrak X_{\mathfrak f} ( \mathcal O)$  with $\psi_g ( \mathfrak p) = \psi_g (\mathfrak p')$, then
\[
\mathfrak p = \mathfrak p D \cap \mathcal O = \psi_g (\mathfrak p) \cap \mathcal O = \psi_g (\mathfrak p') \cap \mathcal O = \mathfrak p' D \cap \mathcal O = \mathfrak p'  \,,
\]
whence $\psi_g$ is injective.

(3) Suppose that $\gamma_{\mathfrak f}^{-1} (g)$ is finite for every $g \in G_{\mathfrak f, \infty} (D)$. By (2), $\gamma_{\mathfrak f, \infty}$ is well-defined and it remains to show surjectivity. Let $g \in G_{\mathfrak f, \infty} (D)$. Let  $G_0 \subseteq G_{\mathfrak f} ( \mathcal O)$ denote the union of classes from $\gamma_{\mathfrak f}^{-1} (g)$ and consider the map
\[
\psi_g \colon g \cap \mathfrak X_{\mathfrak f} (D) \to G_0 \cap \mathfrak X_{\mathfrak f} (\mathcal O) \,, \quad \text{defined by} \quad \mathfrak P \longmapsto \mathfrak P \cap \mathcal O \,.
\]
Let $\mathfrak P \in g \cap \mathfrak X_{\mathfrak f} (D)$ and $\mathfrak p = \mathfrak P \cap \mathcal O$. Then $\mathfrak p \in \mathfrak X_{\mathfrak f} ( \mathcal O)$ and $\mathfrak p D = \mathfrak P$. Thus, we obtain that $[\mathfrak p] \in \gamma_{\mathfrak f}^{-1} (g)$, $\gamma_{\mathfrak f} ( [\mathfrak p]) = [\mathfrak p D] = g$ and $\mathfrak p \in [\mathfrak p] \in \gamma_{\mathfrak f}^{-1} (g)$, whence $\psi_g$ is well-defined. Next we show that $\psi_g$ is injective. If $\mathfrak P, \mathfrak P' \in g \cap \mathfrak X_{\mathfrak f} (D)$ with $\psi_g ( \mathfrak P) = \psi_g ( \mathfrak P')$, then
\[
\mathfrak P = (\mathfrak P \cap \mathcal O)D = \psi_g (\mathfrak P)D = \psi_g (\mathfrak P')D = (\mathfrak P' \cap \mathcal O)D= \mathfrak P' \,,
\]
whence $\psi_g$ is injective.

Since $g \cap \mathfrak X_{\mathfrak f} (D)$ is infinite and $\psi_g$ is injective, it follows that $G_0 \cap \mathfrak X_{\mathfrak f} (\mathcal O)$ is infinite. Since $\gamma_{\mathfrak f}^{-1} (g)$ is finite, there is some $h \in \gamma_{\mathfrak f}^{-1} (g)$ for which $h \cap \mathfrak X_{\mathfrak f} (\mathcal O)$ is infinite. Thus, $h \in G_{\mathfrak f, \infty} ( \mathcal O)$ and $\gamma_{\mathfrak f, \infty} (h) = \gamma_{\mathfrak f} (h)=g$, whence $\gamma_{\mathfrak f, \infty}$ is surjective.
\end{proof}

\smallskip
\begin{corollary} \label{coro-8}
Let $R$  be either equal to the Dedekind domain $D$ of Theorem \ref{con2} or be equal to a localization $D_S$ for a submonoid $S \subseteq D^{\bullet}$, and suppose that its Picard group $\Pic (R)$ is infinite. Let $\mathcal O \subseteq R$ be an order with conductor $\mathfrak f$ such that the factor group of $(R/\mathfrak f)^\times / (\mathcal O/\mathfrak f)^\times$ modulo $R^\times/{\mathcal O}^{\times}$ is finite. Then infinitely many classes of $\Pic (\mathcal O)$ contain infinitely many invertible prime ideals.
\end{corollary}

\noindent
{\it Remark.} The proof of the Corollary uses only Property (3) of Theorem \ref{con2} but it does not make use of the specific construction. Note that every Dedekind domain with the finite norm property satisfies the additional condition on the above mentioned factor group.

\begin{proof}
We use all notation as introduced above. Thus, we need to show that $G_{\mathfrak f, \infty} ( \mathcal O)$ is infinite. By Theorem \ref{con2}, every class of $\Pic (R)$ contains infinitely many prime ideals and hence every class contains infinitely many prime ideals coprime to the conductor. This means that $G_{\mathfrak f, \infty} ( R)$ is infinite. Therefore, it suffices to verify that $\gamma_{\mathfrak f}^{-1} (g)$ is finite for every $g \in G_{\mathfrak f, \infty} (R)$, because  then the assertion follows from Lemma \ref{technical}(3).

Let $g \in G_{\mathfrak f, \infty} (R)$. Consider the exact sequence
\[
\boldsymbol 1 \ \rightarrow \ R^\times/{\mathcal O}^{\times} \ \rightarrow \ (R/\mathfrak f)^\times / (\mathcal O/\mathfrak f)^\times \ \rightarrow \ \Pic (\mathcal O ) \ \overset{\gamma}{\rightarrow} \ \Pic (R) \ \rightarrow \ \boldsymbol 0\,.
\]
By the finiteness of the factor group in the assumption, it follows that $\ker ( \gamma )$ is finite. If $h \in \Pic ( \mathcal O)$ with $\gamma (h) = g$, then $\gamma_{\mathfrak f}^{-1} (g) \subseteq \gamma^{-1} (g) = h + \ker ( \gamma )$, whence $\gamma_{\mathfrak f}^{-1} (g)$ is finite.
\end{proof}

\smallskip
\begin{remark}~
{\em
(1) Let $D$ be a Dedekind domain, $\mathcal O \subseteq D$ be an order, and let $\mathcal O^* \subseteq \mathcal O$ be as in the previous discussion. Since $\mathcal O^*$ is a regular congruence monoid in $D$, every class of $\Cl_v (\mathcal O^*)$ is a union of ray classes. Thus, if every ray class contains infinitely many prime ideals, then every class of $\Cl (\mathcal O^*)$ contains infinitely many prime ideals, whence every class of $\Pic (\mathcal O)$ contains infinitely many prime ideals (see \cite[Proposition 2.11.14]{Ge-HK06a}).

(2) The assumption made in (1) (on prime ideals in ray classes) holds true if $D$ is a holomorphy ring in a global field. However, it does not hold true in general. Indeed, there are Dedekind domains $D$ and orders $\mathcal O \subseteq D$ such that every class of $\Pic (D)$ contains infinitely many prime ideals but $\Pic (\mathcal O)$ does not have this property (see \cite[Remark 3.9]{Fa-Wi22c}).}
\end{remark}

\medskip
Our second corollary deals with the arithmetic of the Dedekind domains occurring in our main result (Theorem \ref{con2}). In order to do so, we gather the involved arithmetic concepts.

Let $M$ be a  monoid. If $a \in M$ and $a = u_1 \cdot \ldots \cdot u_k$, where $k \in \N$ and $u_1, \ldots, u_k$ are irreducible elements of $M$, then $k$ is called a factorization length of $a$. The set  $\mathsf L (a)$ of all factorization lengths of $a$ is called the {\it set of lengths} of $a$. It is convenient to set $\mathsf L (a) = \{0\}$ if $a$ is invertible. Note that
\begin{itemize}
\item $\mathsf L (a) = \{0\}$ if and only if $0 \in \mathsf L (a)$ if and only if $a$ is invertible.

\item $\mathsf L (a) = \{1\}$ if and only if $1 \in \mathsf L (a)$ if and only if $a$ is irreducible.
\end{itemize}
If $M$ is $v$-noetherian, then every non-unit has a factorization into irreducibles and all sets of lengths are finite.
We let
\[
\mathcal L (M) = \{\mathsf L (a) \mid a \in M\}
\]
denote the {\it system of sets of lengths} of $M$. For a finite nonempty set $L = \{m_0, \ldots, m_k\}$ $\subseteq \Z$, with $k \in \N_0$ and $m_0 < \ldots < m_k$, we denote by $\Delta (L) = \{m_i - m_{i-1} \mid i \in [1,k] \} \subseteq \N$ the set of distances of $L$. Then
\[
\Delta (M) = \bigcup_{L \in \mathcal L (M)} \Delta (L) \ \subseteq \N
\]
denotes the {\it set of distances} of $M$.

For an additive abelian group $G$ and a subset $G_0 \subseteq G$, let $\mathcal F (G_0)$ denote the free abelian monoid with basis $G$. If $S = g_1 \cdot \ldots \cdot g_{\ell} \in \mathcal F (G_0)$, then $\sigma (S) = g_1 + \ldots + g_{\ell}$ is the sum of $S$. The set
\[
\mathcal B (G_0) = \{ S \in \mathcal F (G_0) \mid \sigma (S) = 0 \} \subseteq \mathcal F (G_0)
\]
is a submonoid of $\mathcal F (G_0)$ and, since the inclusion $\mathcal B (G_0) \hookrightarrow \mathcal F (G_0)$ is a divisor homomorphism, $\mathcal B (G_0)$ is a Krull monoid. In additive combinatorics, $\mathcal F (G_0)$ is called the monoid of {\it sequences } over $G_0$ and $\mathcal B (G_0)$ is the {\it monoid of zero-sum sequences} over $G_0$. Furthermore,
\[
\Delta^* (G) = \{ \min \Delta ( \mathcal B (G_0) ) \mid G_0 \subseteq G \ \text{with} \ \Delta ( \mathcal B (G_0) ) \ne \emptyset \}
\]
is the {\it set of minimal distances} of $\mathcal B (G)$. If $G$ is finite with $|G| \ge 3$, then $\Delta^* (G)$ is finite and, by \cite{Ge-Zh16a}, we have
\[
\max \Delta^* (G) = \max \{\mathsf r (G)-1, \exp (G)-2\} \,,
\]
where $\exp (G)$ is the exponent of $G$ and $\mathsf r (G)$ is the  rank of $G$ (i.e., the maximum of the $p$-ranks of $G$).

A subset $L \subseteq \Z$ is called an
{\it almost arithmetic multiprogression} ({\rm AAMP})  with  {\it difference}  $d$ and  {\it bound}  $M$ if
\[
L = y + (L' \cup L^* \cup L'') \, \subseteq \, y + \mathcal D + d \Z \,,
\]
where
\begin{itemize}
\item $d \in \N$ and $\{0,d\} \subseteq \mathcal D \subseteq [0,d]$,

\item $\min L^* = 0$ and $L^* = (\mathcal D + d \Z) \cap [0, \max L^*]$, and

\item $L' \subseteq [-M, -1]$,  $L'' \subseteq \max L^* + [1,M]$,  and $y \in \Z$.
\end{itemize}

\smallskip
Before we formulate our final corollary, we draw attention to an interesting special case.
Suppose that $(G_i)_{i \in \N}$ is the family of finite cyclic groups of order $|G_i|=i$ for all $i \in \N$.
Then, by Theorem \ref{con2}, there is a Dedekind domain $D$ such that (i) $\Pic(D) \cong \bigoplus_{i \in \N}G_i$, which is infinite, and
(ii) for every $i \in \N$, there is a submonoid $S_i$ of $D^{\bullet}$ so that $\Pic (D_{S_i}) \cong G_i$.
Hence, all the cyclic groups $G_i$ and their associated sets of minimal distances $\Delta^* (G_i)$ can be realized as the Picard groups and as sets of minimal distances of overrings of the fixed Dedekind domain $D$, and all domains have infinitely many prime ideals in all classes of their respective class groups.
The sets $\Delta^* (G_i)$ have found much interest in recent literature (see \cite{Pl-Sc20a}).

\smallskip
\begin{corollary} \label{coro-9}
Let $R$ be either equal to the  Dedekind domain $D$ of Theorem \ref{con2}
or equal to a localization $D_S$ for a submonoid $S \subseteq D^{\bullet}$.
\begin{enumerate}
\item Suppose that $\Pic (R)$ is finite. Then there is an $N \in \N_0$ with the following property: for every $L \in \mathcal L (R^{\bullet})$ there is some $d$ in the finite set $\Delta^* \big( \Pic (R) \big)$ such that $L$ is an \AAMP \ with difference $d$ and bound $N$.

\smallskip
\item Suppose that $\Pic (R)$ is infinite. If $L = \{m_1, \ldots, m_k \} \subseteq \N_{\ge 2}$ with $k \in \N$, $2 \le m_1 < \ldots < m_k$, and $n_1, \ldots, n_k \in \N$, then there is an $a \in R^{\bullet}$ which has at least $n_i$ distinct factorizations of length $m_i$ for all $i \in [1,k]$ and no factorizations of other lengths.
     In particular,
      \[
      \mathcal L (R^{\bullet}) = \Big\{ \{0\}, \{1\} \Big\} \cup \Big\{ L \subseteq \N_{\ge 2} \mid L \ \text{is finite nonempty} \Big\} \,.
      \]
\end{enumerate}
\end{corollary}

\begin{proof}
Note that the multiplicative monoid $R^{\bullet}$ of non-zero elements is a Krull monoid. Let $M$ be a Krull monoid with class group $G$ and let $G_0 \subseteq G$ denote the set of classes containing prime divisors. Then there is a transfer homomorphism $\boldsymbol \beta \colon M \to \mathcal B (G_0)$, which implies that $\mathsf L_M (a) = \mathsf L_{\mathcal B (G_0)} ( \boldsymbol \beta (a))$ and hence $\mathcal L (M) = \mathcal L \big( \mathcal B (G_0) \big)$ (\cite[Theorem 3.4.10]{Ge-HK06a}). By Theorem \ref{con2}, every class of $\Pic (R)$ contains infinitely many prime ideals, whence we have a transfer homomorphism $\theta \colon R^{\bullet} \to \mathcal B \big( \Pic (R) \big)$.

(1) Since $\mathcal L (R^{\bullet}) = \mathcal L \big( \mathcal B (\Pic (R)) \big)$ and $\mathcal B (\Pic (R))$ is finitely generated (here we use that $\Pic (R)$ is finite), the assertion follows from \cite[Theorem 4.4.11]{Ge-HK06a}.

(2) By \cite[Theorem 7.4.1]{Ge-HK06a}, there is an element $A \in \mathcal B \big( \Pic (R) \big)$ with the required properties. Then, there is an element $a \in R^{\bullet}$ with $\boldsymbol \beta (a) = A$. Since $\boldsymbol \beta$ is a transfer homomorphism, we have $\mathsf L_R (a) = \mathsf L_{\mathcal B \big( \Pic (R) \big)} (A)$ and, for every $i \in [1,k]$, the number of factorizations of $a$ of length $m_i$ is greater than or equal to the number of factorizations of $A$ of length $m_i$.
\end{proof}

Corollary \ref{coro-9}.(1) is a key result on the arithmetic of Krull monoids with finite class group having prime divisors in all classes. These Krull monoids are studied in detail with methods from additive combinatorics (for surveys see \cite{Ge-Ru09, Sc16a}).
The arithmetic of Krull monoids with finitely generated class group (without any assumption on the distribution of prime divisors in the classes) is studied  in the recent monograph \cite{Gr22a}. The statement in Corollary \ref{coro-9}.(2) heavily depends on the fact that every class of the Picard group contains at least one prime divisor. We highlight this in the following remark.

\smallskip
\begin{remark}
{\em For every abelian group $G$, that is  a direct sum of cyclic groups  there is a Dedekind domain $D$ with $\Pic (D) \cong G$ and
\[
\mathcal L (D^{\bullet}) = \Big\{ \{k\} \mid k \in \N_0 \Big\} \,.
\]
In more technical terms, every abelian group, that is a direct sum of cyclic groups, has a half-factorial generating set (see \cite[Proposition 3.7.9]{Ge-HK06a}). The standing conjecture says that this is true for all abelian groups (\cite{Gi06a}) and the conjecture has been confirmed for all Warfield groups (\cite{Lo98,Ge-Go03}).}
\end{remark}

\bigskip
\noindent
{\bf Acknowledgement.} We thank Andreas Reinhart for many helpful discussions. This work was supported by the Basic Science Research Program through the National Research Foundation of Korea (NRF)
funded by the Ministry of Education (2017R1D1A1B06029867), and by the Austrian Science Fund FWF, Project Number P33499.

\vspace{.2cm}
\noindent

\providecommand{\bysame}{\leavevmode\hbox to3em{\hrulefill}\thinspace}
\providecommand{\MR}{\relax\ifhmode\unskip\space\fi MR }
\providecommand{\MRhref}[2]{%
  \href{http://www.ams.org/mathscinet-getitem?mr=#1}{#2}
}
\providecommand{\href}[2]{#2}

\end{document}